\documentclass[12pt]{amsart}
\usepackage{amssymb,verbatim}
\usepackage{graphicx}

\newtheorem{theorem}{Theorem}[section]

\newtheorem{corollary}[theorem]{Corollary}
\newtheorem{lemma}[theorem]{Lemma}

\theoremstyle{definition}

\newcommand{\ol}{\overline}
\newcommand{\sm}{\setminus}

\newcommand{\Complex}{\mathbb{C}}

\newcommand{\0}{\emptyset}

\newcommand{\bd}{\mbox{Bd}}
\newcommand{\dia}{\mbox{diam}}
\newcommand{\sh}{\mbox{Sh}}
\newcommand{\ch}{\mbox{Ch}}
\newcommand{\Ga}{\Gamma}
\newcommand{\e}{\varepsilon}
\newcommand{\al}{\alpha}
\newcommand{\ph}{\varphi}
\newcommand{\be}{\beta}
\newcommand{\ga}{\gamma}
\newcommand{\si}{\sigma}
\newcommand{\ta}{\theta}
\newcommand{\om}{\omega}

\newcommand{\la}{\lambda}
\newcommand{\nin}{\not\in}
\newcommand{\imp}{\mbox{Imp}}

\newcommand{\bs}{\bar{s}}

\newcommand{\hsi}{\hat{\sigma}}

\newcommand{\C}{\mathbb{C}}

\newcommand{\D}{\mathbb{D}}

\newcommand{\bbd}{\mbox{$\mathbb{D}$}}
\newcommand{\disk}{\mathbb{D}}

\newcommand{\ucirc}{\mathbb{S}^1}
\newcommand{\uc}{\mathbb{S}^1}

\newcommand{\lam}{\mathcal{L}}

\newcommand{\A}{\mathcal{A}}
\newcommand{\K}{\mathcal{K}}
\newcommand{\Hc}{\mathcal{H}}

\newcommand{\Ta}{\Theta}
\newcommand{\iy}{\infty}
\newcommand{\cut}{\mathrm{Cut}}

\begin{document}

\date{May 7, 2008}
\title[Basic uniCremer polynomials of arbitrary degree]
{The Julia sets of basic uniCremer polynomials of arbitrary degree}

\author{Alexander~Blokh}

\thanks{The first author was partially
supported by NSF grant DMS-0456748}

\author{Lex Oversteegen}

\thanks{The second author was partially  supported
by NSF grant DMS-0405774}

\address[Alexander~Blokh and Lex~Oversteegen]
{Department of Mathematics\\ University of Alabama at Birmingham\\
Birmingham, AL 35294-1170}

\email[Alexander~Blokh]{ablokh@math.uab.edu}
\email[Lex~Oversteegen]{overstee@math.uab.edu} 


\begin{abstract} Let $P$ be a polynomial of degree $d$ with
a Cremer point $p$ and no repelling or parabolic periodic bi-accessible points.
We show that there are two types of such Julia sets
$J_P$. The \emph{red dwarf} $J_P$ are nowhere connected im kleinen and such
that the intersection of all impressions of external angles is a continuum
containing $p$ and the orbits of all critical images. The
\emph{solar} $J_P$ are such that every angle with dense orbit has a degenerate
impression disjoint from other impressions and $J_P$ is connected im kleinen
at its landing point. We study bi-accessible points and locally connected models
of $J_P$ and show that such sets $J_P$ appear through polynomial-like maps for generic polynomials
with Cremer points. Since  known tools break down for $d>2$
(if $d>2$ it is not known if there are \emph{small cycles} near $p$ while if $d=2$
this result is due to Yoccoz), we introduce \emph{wandering
ray continua} in $J_P$ and provide a new application of  \emph{Thurston laminations}.
\end{abstract}

\maketitle

\section{Introduction and description of the results}\label{intro}

Consider a degree $d$ polynomial map $P:\C\to \C$ of the complex plane $\C$;
denote the Julia set of $P$ by $J_P=J$, the filled-in Julia set by $K_P=K$ and
the basin of infinity by $A_\iy$. Assume that $J$ is connected. Call an
irrational neutral periodic point a \emph{CS-point}. In his fundamental paper
\cite{kiwi97} Kiwi showed that if $P$ has no CS-points then $P|_J$ is modeled
by an \emph{induced map on a topological (locally connected) Julia set} and $J$ is
locally connected at its preperiodic points. However in the case when $P$ has
CS-points similar results are not known, and there are few results about the
topological structure of $J$.

We address this issue and consider \emph{basic uniCremer} polynomials $P$
defined as those with a Cremer periodic point and without repelling (parabolic)
periodic bi-accessible points (by \cite{gm93, kiwi00} then the Cremer point is
unique and fixed). We show that then for $J$ either (a) all impressions are
``big'' (contain an invariant continuum) and $J$ is nowhere \emph{connected im
kleinen} (topological notion close to local connectedness \cite{m62}), or (b)
there are plenty of degenerate impressions with dense orbits in $J$ at each of
which $J$ is connected im kleinen. This dichotomy in a way takes further a
standard dichotomy between non-locally connected and locally connected Julia
sets.

Similar results for $d=2$ were obtained in \cite{bo06}. However
in \cite{bo06} we strongly rely on \cite{grismayeover99} and Yoccoz's small
cycles \cite{yoc95} none of which is  known in general. The novelty of our
present approach to polynomials of \emph{arbitrary} degrees is that we use
Thurston laminations (previously not used in this setting), and rely on our new
fixed point result obtained in \cite{bo08}. This together with a new argument leads to the
generalization of the results of \cite{bo06} and to new results concerning
bi-accessible points of a basic uniCremer Julia set $J$ and the fact that
any locally connected quotient spaces of $J$ must be degenerate.

The study of basic uniCremer Julia sets is justified  by two facts (see Section~\ref{sketch}):
(1) all polynomials
with Cremer points but without critical preperiodic and parabolic points
contain ``basic uniCremer polynomial-like Julia sets'' as subsets; (2) a
polynomial with one critical point and a Cremer fixed point must be basic
uniCremer.

Let us pass on to the precise statements. It is known that there exists a
conformal isomorphism $\Psi$ from the complement of the closure of the open
unit disk $\bbd$ onto the complement of $K$. Moreover, $\Psi$ conjugates
$z^d|_{{\C\sm \ol{\D}}}$ and $P|_{{\C\sm K}}$. The $\Psi$-image $R_\al$ of the
radial line of angle $\al$ in $\C\sm \D$ is called an \emph{(external) ray}. If
$J$ is locally connected, $\Psi$ extends to a continuous function $\ol{\Psi}$
which semiconjugates $z^d|_{{\C\sm \D}}$ and $P|_{\ol{A_\iy}}$. Set
$\psi=\ol{\Psi}|_{\ucirc}$ and define an equivalence relation $\sim_P$ on
$\ucirc$ by $x \sim_P y$ if and only if $\psi(x)=\psi(y)$. The equivalence
$\sim_P$ is called the \emph{($d$-invariant) lamination (generated by $P$)}.
The quotient space $\ucirc/\sim_P=J_{\sim_P}$ is homeomorphic to $J$ and the
map $f_{\sim_P}:J_{\sim_P}\to J_{\sim_P}$ induced by $z^d|_{\uc}\equiv \si$ is
topologically conjugate to $P|_J$. The set $J_{\sim_P}$ (with the map
$f_{\sim_P}$) is a model of $P|_J$, called the \emph{topological Julia set}.

In his fundamental paper \cite{kiwi97} Kiwi extended this to polynomials $P$
with connected Julia set and no irrational neutral periodic points (called
\emph{CS-points}) for which he obtained a $d$-invariant lamination $\sim_P$ on
$\ucirc$ with $P|_J$ semi-conjugate to the induced map
$f_{\sim_P}:J_{\sim_P}\to J_{\sim_P}$ by a monotone map $m:J\to J_{\sim_{P}}$
(\emph{monotone} means a map with connected point preimages). Kiwi also proved
that for all periodic points $p\in J$ the set $J$ is locally connected at $p$
and $m^{-1}\circ m(p)=\{p\}$. However in some cases the results of
\cite{kiwi97} do not apply. E.g., take a quadratic polynomial $P$ with a
\emph{Cremer fixed point} (i.e. with a neutral non-linearizable fixed point
$p\in J$ such that $P'(p)=e^{2\pi i \al}$ with $\al$ irrational). By Sullivan
\cite{sul83}, the Julia set $J_P$ is not locally connected. Moreover, since $P$
has a CS-point, Kiwi's results from \cite{kiwi97} do not apply.

As mentioned above, we addressed this issue for quadratic polynomials with
fixed Cremer point in \cite{bo06}. Yet the tools from \cite{bo06} have a
limited, ``quadratic'', nature. Most importantly, \cite{bo06} makes an
essential use of \cite{grismayeover99} where so-called \emph{building blocks}
were studied. However, the \cite{grismayeover99} applies only in the quadratic
case: it is based in particular upon the existence of \emph{small cycles} near
the Cremer point \cite{yoc95}, and as of yet, it is not known for polynomials
of higher degrees whether small cycles near Cremer points exist or not. Thus,
the building blocks machinery breaks down for higher degrees and needs to be
replaced.

Also, the main topological result in \cite{bo06} is Theorem 2.3 which says that
\emph{if $K$ is a non-separating $g^n$-invariant plane continuum/point not
containing a Cremer point of a polynomial $g$, $g^n|_K$ is a
\textbf{homeomorphism} in an open neighborhood $V\supset K$, and some technical conditions
are satisfied, then $K$ is a point}. It sufficed to consider a homeomorphism in
\cite{bo06} because if $d=2$ ($g$ is quadratic) then $p\nin K$ implies $c\nin
K$ (the critical point $c$ is recurrent and $p\in \om(c)$ \cite{mane93}). If
$d>2$, there are other critical points, hence one needs to generalize Theorem
2.3 of \cite{bo06} to non-homeomorphisms.

Thus, to study polynomials of higher degrees with Cremer periodic points new
tools are needed. Here we suggest a different approach allowing us to extend
the results of \cite{bo06} (by \cite{sz99} quadratic polynomials with a fixed
Cremer point are basic uniCremer, so the present paper extends \cite{bo06}).
Namely, we use Thurston laminations, the notion of a \emph{wandering ray
continuum}, and a new topological result from \cite{bo08}, and generalize
\cite{bo06} to basic uniCremer polynomials of arbitrary degree. Also, we
obtain such results as Theorem~\ref{nowander} (concerning certain wandering
sets and bi-accessible points in basic uniCremer Julia sets) and
Theorem~\ref{nomonproj} (which shows that \emph{any} monotone map of $J$ onto a
locally connected continuum collapses $J$ to a point).

To proceed we need a few definitions. Let  $\Psi$ be a conformal isomorphism
from the complement of the closure of the open
unit disk $\bbd$ onto the complement of $K$. We identify points on the unit circle with their
argument $\al\in[0,2\pi)$ and call them \emph{angles}. For each angle $\al\in\ucirc$,
the \emph{impression of $\al$}, denoted
by $\imp(\al)$, is the set $\{w\in \C\mid \text{there exists a sequence } z_i\to\al
\text{ such that } \lim \Psi(z_i)=w\}$. Then $\imp(\al)$ is a subcontinuum of $J$.

If the impression $\imp(\al)$ is disjoint
from all other impressions then the angle $\al$ is called \emph{K-separate}. A
continuum $X$ is \emph{connected im kleinen at a point $x$} \cite{m62} provided
for each open set $U$ containing $x$ there exists a connected set $C\subset U$
with $x$ in the interior of $C$. Now we can state our main theorem which generalizes to
basic uniCremer polynomials the results of \cite{bo06}.

\setcounter{section}{5} \setcounter{theorem}{9}

\begin{theorem}\label{maina1} For a basic uniCremer polynomial $P$ the following
facts are equivalent:

\begin{enumerate}

\item there is an impression not containing the Cremer point;

\item there is a degenerate impression;

\item the set $Y$ of all K-separate angles with degenerate
impressions contains all angles with dense orbits and a dense set of periodic
angles, and the Julia set $J$ is connected im kleinen at the landing points of
the corresponding rays;

\item there is a point at which the Julia set is connected im
kleinen.

\end{enumerate}

\end{theorem}

Basic uniCremer Julia set with the properties from Theorem~\ref{maina1} are
said to be \emph{solar}. The remaining basic uniCremer Julia sets are called
\emph{red dwarf} Julia sets. They can be defined as basic uniCremer Julia sets
such that all impressions contain $p$. These notions have been introduced in
\cite{bo06} and were further studied in \cite{bbco07} where it was shown that solar Julia sets
of degree two, with positive Lebesgue measure, exist. The following lemma describes
red dwarf Julia sets and complements Theorem~\ref{maina1}.

\setcounter{section}{5} \setcounter{theorem}{2}

\begin{lemma}\label{redwarf1} If $J$ is a red dwarf Julia set then the (non-empty)
intersection of all impressions contains all forward images of all critical
points, there exists $\e>0$ such that the diameter of any impression is greater
than $\e$, and there are no points at which $J$ is connected im kleinen.
Moreover, in this case no point of $J$ is bi-accessible and $p$ is not
accessible from $\C\setminus J$.
\end{lemma}

We also prove Theorem~\ref{nowander1} which partially extends to basic
uniCremer polynomials results of Schleicher and Zakeri \cite{sz99} related to
the question of McMullen \cite{mcm94} as to whether a quadratic Julia set with
a fixed Cremer point has bi-accessible points. Given a set $B$ of angles, set
$\imp(B)=\cup_{\be\in B} \imp(\be)$ ($\Pi(B)=\cup_{\be\in B} \Pi(\be)$
($\imp(\al)$ is the impression of an angle $\al$ and $\Pi(\al)=\ol{R_\al}\sm R_\al$ is the
\emph{principal} set of an angle $\al$). Call $X$ a \emph{ray continuum} if $X$ is a
continuum or a point and for some set of angles $B$ we have $\Pi(B)\subset
X\subset \imp(B)$. The maximal set $B$ of angles such that $\Pi(B)\subset
X\subset \imp(B)$ is said to be \emph{connected} to $X$. A set is said to be
\emph{wandering} if all its images are pairwise disjoint.

\setcounter{section}{4} \setcounter{theorem}{2}

\begin{theorem}\label{nowander1}
Suppose that $J$ is the Julia set of a basic uniCremer polynomial $P$ of degree
$d$. Suppose that $K'\subset J$ is a wandering ray continuum connected to a set
of angles $\A'$ with $|\A'|\ge 2$. Then there exists an $n$ such that $|\si^n(\A')|$ is a
singleton and for every $\ta\in \A'$, $\cup_{n\ge 0} \si^n(\ta)$ is not dense
in $\ucirc$. In particular, $K'$ is pre-critical, and if a point $x\in J$ is
bi-accessible then it is either precritical or preCremer.
\end{theorem}

In Theorem~\ref{nomonproj1} we show that Kiwi's results \cite{kiwi97} do not
apply to basic uniCremer polynomials for principal reasons: any monotone map of
a basic uniCremer Julia set must collapse the whole set to a point.

\setcounter{section}{5} \setcounter{theorem}{1}

\begin{theorem}\label{nomonproj1} Suppose that $P$ is a basic uniCremer polynomial
and $\ph:J\to A$ is a monotone map of $J$ onto a locally connected continuum
$A$. Then $A$ is a singleton.
\end{theorem}

The paper is designed as follows. In Section~\ref{sketch} we motivate our work
by showing that all polynomials with Cremer points, but without critical preperiodic
and parabolic points, contain ``basic uniCremer polynomial-like Julia sets'' as
subsets and that a polynomial with one critical point and a Cremer fixed point
must be basic uniCremer. In Section~\ref{fip} we develop the topological
tools. In Section~\ref{wandecon} we develop the  laminational tools and
use them to study some wandering continua in the Julia set of a basic uniCremer
polynomial. Finally, in Section~\ref{quadra} we obtain our main results.

\setcounter{section}{1} \setcounter{theorem}{0}

\section{Motivation}\label{sketch}

In this section we suggest additional motivation for our research and provide a
criterion for a polynomial to be basic uniCremer. Applying renormalization
ideas that allow one to find ``little" Julia sets (of appropriate
polynomial-like maps) inside ``big" ones, we show that \emph{if a polynomial
$P$, with connected Julia set $J$, has no critical preperiodic
and parabolic points, then a periodic Cremer point $p$ of $P$ belongs to a polynomial-like
Julia set, contained in $J$, corresponding to a basic uniCremer polynomial}. Thus, basic
uniCremer Julia sets appear in Julia sets of \emph{almost all} polynomials with
Cremer points as Julia sets of polynomial-like maps. Similar constructions,
based upon \cite{dh85b}, can be found, e.g., in \cite{hub93, lev98, sch04}. We
also sketch the proof of the fact that if no critical point of a polynomial $P$
with a fixed Cremer point $p$ is separated from $p$ by means of repelling
periodic bi-accessible points and their preimages in $J_P$ then $P$ is basic
uniCremer. Finally, we show that a polynomial with unique critical point and
fixed Cremer point is basic uniCremer.

Let $p$ be a fixed Cremer point of a polynomial $P$ with connected Julia set $J$.
Denote the set of all periodic repelling bi-accessible points of $P$ by $B'\ne
\0$. Consider the set $B$ of \emph{all} preimages of points of $B'$. The rays
landing at any $x\in B$ partition $\C$ into the wedges. Choose the closed wedge
$W_x$ containing $p$ and denote its boundary rays by $R_{\al_x}$ and
$R_{\be_x}$. Connect $\al_x$ and $\be_x$ with a chord $\ell=\ol{\al_x\be_x}$
inside $\bbd$ and consider the closure $T_x$ of the part of $\bbd\sm \ell$
which corresponds to the angles with rays inside $W_x$. In this case call
$\ell$ an \emph{admissible} chord, set $x=z_\ell$ and call
$\cut_\ell=\ol{R_{\al_x}\cup R_{\be_x}}$ an \emph{admissible} cut. If $X\subset
B$ then we denote by $W_X$ the intersection of all wedges $W_x$, $x\in X$.
Consider the intersection $W_B=W$ of  wedges $W_x$ and the intersection $T$ of
sets $T_x$ over \emph{all} points $x\in B$. For a chord $\ell=\ol{\al\be}\subset
\partial T$ let $I_\ell=(\al, \be)\subset\ucirc$ be the open arc disjoint from $T$; call $I_\ell$
the \emph{prime arc} of $\ell$.

We show that $W\cap J_P$ is the connected Julia set of a
polynomial-like map. It is easy to see that 1) $W\cap J_P$ is an invariant
continuum, 2) $T$ is a subcontinuum of $\bbd$, and 3) $\partial T$ consists of
points in $\partial T\cap \uc$ (nowhere dense in $\uc$) and a collection $\K$ of pairwise
disjoint chords (except perhaps for the endpoints). If we connect the
$\si$-images of the endpoints of a chord $\ell$ from $\K$, we get a chord from
$\K$ understood as the ``image'' $\hsi(\ell)$ of $\ell$. Let $\Hc$ be the
family of all $\hsi$-periodic chords (in this sense) in $\K$. Then a) $\Hc$ is
finite and \emph{can be assumed to consist only of chords with fixed
endpoints}, b) there are no critical chords in $\K$, c) all chords in $\K$
eventually map onto chords of $\Hc$. In particular, $\Hc$ is not empty.

Let $\A=\{A_1=(\al_1, \be_1), \dots, A_n=(\al_n, \be_n)\}$ be \emph{all}
components of $\uc\sm T$ longer than $\frac 1d$ with $\ell_1=\ol{\al_1\be_1},
\dots, \ell_n=\ol{\al_n\be_n}$. The family $\A$ is non-empty, e.g. because for
any chord $\ell\in \Hc$ the arc $I_\ell$ is longer than $\frac 1d$. Moreover,
by the construction there exists $i$ and an arc $B_i\subset A_i$ with endpoints $\ta,
\ga$ such that $B_i$ is longer than $\frac 1d$ and $\ol{\ta\ga}$ is admissible.
Then the part of the plane separated from $W$ by $\ol{R_\ta\cup R_\ga}$ must
contain a critical point. This implies one of our claims: \emph{if no critical point is
separated in the Julia set from a fixed Cremer point by a point of $B$ then the
polynomial is basic uniCremer.}

On the other hand, it is easy to see that if $p$ is a Cremer fixed point then
there must exist a critical point $c$ which is not separated from $p$ by points
of $B$. Indeed, otherwise there is a set of points $X=\{x_1, \dots, x_m\}$ and
a set of rays  landing at them  $Z=\{\tilde{R'}_1, \tilde{R''}_1, \dots,
\tilde{R'}_m, \tilde{R''}_m\}$ such that the component $M$ of $\C\sm \cup_i
(\tilde{R'}_i\cup \tilde{R''}_i)$ containing $p$ contains no critical points.
Choose an equipotential curve $L$ intersecting the rays at points close to the
appropriate points of $X$. Then construct small arcs connecting points of
intersection of $L$ with $\tilde{R'}_i, \tilde{R''}_i$ outside $M$. The thus
constructed curve $S$ encloses a simply connected region $U$ whose closure
contains no critical points. By \cite{pere94, pere97} then the so-called
\emph{hedgehog} $H\subset J$ created inside $U$ will reach out to the boundary
of $U$. This implies that at least one point of $X$ is a cutpoint of $H$
contradicting \cite{pere94, pere97}. Thus, \emph{if a polynomial has a
\textbf{unique} critical point and a \textbf{fixed} Cremer point then it is
basic uniCremer}; this gives another example of basic uniCremer polynomials.

Let us go back to the general case. For a critical point $c\nin W\cap J$ there
are two angles $\ga\ne \ta\nin W$ with $c\in \imp(\ga)\cap \imp(\ta)$ and
$\si(\ga)=\si(\ta)$. Since $c$ is cut off $W$ by admissible cuts, the component
$A_{i_c}$ of $\uc\sm T$ to which $\ga, \ta$ belong, is well-defined. Let $\ell_{i_c}$ be the leaf
in $\K$ whose endpoints coincide with the endpoints of $A_{i_c}$. For each
$i_c$ choose $r_c$ as the smallest such number that $\si^{r_c}(\ell_{i_c})\in
\Hc$. Given $\bs\in \Hc$, let $H_{\bs}$ be the set of all critical points
$c\nin W\cap J$ with $\hsi^{r_c}(\ell_{i_c})=\bs$.

We construct the simple closed curve $S$ needed to find a polynomial-like
map. Let $\bs=\ol{\ta\ga}\in \Hc$ be admissible. Then $R_\ta=R'_1, R_\ga=R''_1$ land at
$z_{\bs}$. Since $z_{\bs}$ is a repelling fixed point we can choose a very
small neighborhood $U_{\bs}=U_1$ of $z_{\bs}$ in which points are repelled from
$z_{\bs}$. Moreover, by  assumption, critical points of $P$ are not mapped
into repelling periodic points. Hence we can choose $U_{\bs}$ so small that
$P^{r_c}(c)\nin U_{\bs}$ for all $c\in H_{\bs}$. Do this for all admissible
$\bs\in \Hc$.

Now consider $\bs\in \Hc$ which is not admissible. For $b\in H_{\bs}$ choose an
admissible chord $\ell_b=\ol{\al\be}$ with $R_{\al}=R'_2, R_{\be}=R''_2$ very
close to $\ell_{i_b}=\ol{\la\xi}$ with $R_{\la}=R'_3, R_{\xi}=R''_3$. We can
choose the closeness so that $\cut_{\ell_b}=\cut_1$ separates $b$ from $W$ and
then consider both $P^{r_b}(\cut_1)$ (which by the construction is located
outside $W$) and $\si^{r_b}(\ell_b)$ (which by the construction is an
admissible chord outside $T$). Then we choose an admissible chord
$\ell'_{\bs}=\ol{\al_{\bs}\be_{\bs}}$ so close to $\bs$ that it separates all
chords $\si^{r_b}(\ell_b), b\in H_{\bs}$ from $T$ (this is possible since $\bs$
is not admissible and hence is a limit of admissible chords).
Figure~\ref{unicfig1} illustrates this constriction.

\begin{figure}\refstepcounter{figure}\label{unicfig1}\addtocounter{figure}{-1}
\begin{center}
\includegraphics[width=3.5truein]{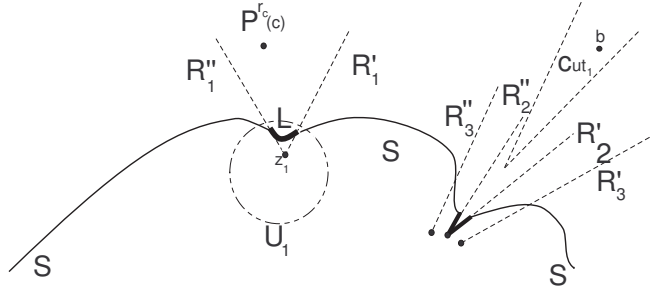}
\caption{Construction of the curve $S$}
\end{center}
\end{figure}

Then choose an equipotential $Q$ which intersects all rays $R_\ta, R_\ga$ taken
over all admissible chords $\bs=\ol{\ta\ga}\in \Hc$ inside $U_{\bs}$ at points
$x_{\bs}$ and $y_{\bs}$. Connect $x_{\bs}$ and $y_{\bs}$ with a small curve
$L_{\bs}=L\subset U$ which is repelled away from $z_{\bs}$ inside the
corresponding wedge between $R_\ta$ and $R_\ga$ and replace the piece of the
equipotential between $R_\ta$ and $R_\ga$ by $L$. Also, this equipotential
intersects all rays $R_{\al_{\bs}}, R_{\be_{\bs}}$ taken over non-admissible
chords $\bs\in \Hc$. In this case we replace the piece of the equipotential $Q$
between the rays $R_{\al_{\bs}}, R_{\be_{\bs}}$ by the union $L_{\bs}$ of
pieces of these rays between the point of their intersection with $Q$ and their
common landing point. Thus, the constructed curve $S$ consists of pieces of the
equipotential $Q$ and either small curves $L_{\bs}$ (taken over all admissible
chords $\bs\in \Hc$) and unions $L_{\bs}$ of pieces of rays $R_{\al_{\bs}},
R_{\be_{\bs}}$ leading to their common landing point (taken over all
non-admissible chords $\bs\in \Hc$).

Denote by $H$ the Jordan disk enclosed by $S$. It follows from the construction
that $P$ maps the curve $S$ outside $\ol{H}$. Then we can consider the pullback
$V$ of $\ol{H}$ under $P$ defined by $P(p)=p$. Clearly, $P:V\to H$ is a
polynomial-like map. Moreover, the construction and standard arguments imply
that the set of all points which never exit $\ol{H}$ under $P$ is exactly $W$:
this can be observed on the disk model first and then transported to the
dynamical plane using the construction, the fact that all periodic points are
repelling, and the fact that so are all periodic angles on the circle. Thus,
$W$ is the filled-in Julia set of the polynomial-like map $P:V\to H$, and
$J\cap W$ is its Julia set.

Denote by $f$ the polynomial to which $P|_V$ corresponds. By \cite{dh85b} there
is a quasi-conformal conjugacy $\psi$ between $P|_{J\cap W}$ and $f|_{J_f}$
(we will use terms ``$P$-ray'' and ``$f$-ray'' in the obvious sense). Notice
that $J_f$ cannot have Fatou domains (or other Cremer points) because by
\cite{gm93, kiwi00} all such sets are cut off $W$. Hence $J_f$ is a
non-separating continuum. It is easy to see that $q=\psi(p)$ is Cremer: there are
arbitrarily small invariant subcontinua in $J\cup W$ containing $p$ (hedgehogs
\cite{pere94, pere97}), hence the corresponding continua exist for $f$ which
excludes that $q$ is repelling.

Suppose that $f$ has a periodic repelling cutpoint $a=\psi(b)$ and set
$\psi(p)=q$. Take $m$ so that $f^m(a)=a$. Then $P^m(b)=b$. By the construction
of $W$ and $T$, there exists an $P^m$-invariant $P$-ray landing at $b$. Then
its $\psi$-image will be an invariant (in a neighborhood of $a$) curve in the
complement of $J_f$ landing at $a$. This implies that all $f$-rays landing at
$a$ must be $f^m$-invariant too (there is no rotation around $a$ in the
$f$-plane). Denote by $A_m$ all $f^m$-fixed points; similarly we see that each
such point (except perhaps for $q$) has an $f^m$-invariant $f$-ray landing at
it. Consider the closure $K$ of a component of $J_f\sm \{A_m\}$ not containing
$q$. Then by \cite{gm93} (or by Corollary~\ref{degimp1} applied to $K$) we see
that there exists an $f^m$-fixed point in $K$ which does not belong to $A_m$, a
contradiction.

\section{Fixed points and impressions}\label{fip}

We begin by describing a corollary of \cite{bo08} which serves
as our major topological tool. Let $X\subset K$ be a non-separating
continuum or a point such that:

\begin{enumerate}

\item Pairwise disjoint continua/points $E_1\subset X, \dots, E_m\subset X$
and finite sets of angles $A_1=\{\al^1_1, \dots, \al^1_{i_1}\}, \dots, A_m=\{\al^m_1,
\dots, \al^m_{i_m}\}$ are given.

\item We have $\Pi(A_j)\subset E_j\subset \imp(A_j)$ (so the set
$E_j\cup (\cup^{i_j}_{k=1} R_{\al^j_k})=E'_j$ is closed and connected).

\item $X$ intersects a unique component $A$ of $\C\sm \cup E'_j$
and $X$ is equal to the union of all $E_j$ and $K\cap A$.

\end{enumerate}

We call such $X$ a \emph{general puzzle-piece} and call the continua $E_i$ the
\emph{exit continua} of $X$. For each $j$, the set $E'_j$ divides the plane
into $i_j$ open  wedges; denote by $W_j$ the one which contains $X\sm E_i$.
Also, let $H_j$ be the chord connecting the angles whose rays form a part of
the boundary of $W_j$, and let $T_j$ be the component of $\disk\sm H_i$ taken
from the side corresponding to $X$ (so that angles whose rays have impressions
in $X$ are belong to $T_j$). Then the intersection $T$ of all sets $T_j$ is a
convex subset of $\disk$ whose boundary consists of chords $H_j$ and some arcs
of the circle. It follows that the impressions of all angles in $T\sm \cup A_i$
are contained in $X$.

\begin{corollary}\cite{bo08}\label{degimp1}
Suppose that for a non-separating continuum $X$ or a point
at least \emph{one} of the following holds.

\begin{enumerate}

\item\label{1} $X$ is a general puzzle-piece with exit continua $E_1,
\dots, E_m$ such that for every $i$ either $E_i$ is a fixed point or
$P(E_i)\subset W_i$.

\item \label{2}$X\subset J$ is an invariant continuum.

\end{enumerate}

If all fixed points which belong to $X\cap J$ are repelling and all rays
landing at them are fixed then  $X$ is a repelling fixed point. Hence, if
$X\subset K$ is a non-degenerate continuum which satisfies (\ref{1}) or
(\ref{2}), then either $X$ contains a non-repelling fixed point or $X$ contains
a repelling fixed point at which a non-fixed ray lands.
\end{corollary}

By Corollary~\ref{degimp1} if $P$ is a basic uniCremer polynomial then
a $P^m$-invariant continuum or a point $K\subset J$, not-containing
the Cremer point $p$ of $P$, is a point.

We now prove a few lemmas which can be of independent interest.
Mostly they deal with  topological properties of impressions. In this
subsection $X$ is a non-separating one-dimensional plane continuum. The
assumption that $X$ is non-separating and one-dimensional is not essential but simplifies the
arguments (e.g., in this case each subcontinuum of $X$ is also a
non-separating, one-dimensional plane continuum and the intersection of any two
subcontinua of $X$ is connected). Speaking of \emph{points} we mean points in
the (dynamic) plane while \emph{angles} mean arguments of external rays.

A \emph{crosscut} $C$ of $K$ is an open arc $C\subset \Complex\sm K$ whose closure is a closed arcs
with endpoints in $K$.  By $Sh(C)$, the \emph{shadow of $C$}, we denote the bounded component
of $\Complex\sm [K\cup C]$. Let $R_\al$ be
an external ray and $\{C_n\}$ a family  of crosscuts all of  whom cross $R_\al$
essentially and are such that $\dia(C_n)\to 0$. Then  it is well known that the impression
 $\imp(\al)=\cap_n \ol{Sh(C_n)}$ .  A continuum $K$ is said to be
\emph{decomposable} if there exist two \emph{proper} subcontinua $A, B\subsetneqq
K$ such that $A\cup B=K$ and \emph{indecomposable} otherwise.
Theorem~\ref{indec} holds for \emph{all} polynomials.

\begin{theorem}\cite[Theorem 1.1]{cmr05}\label{indec} The Julia set of a polynomial $P$ is
indecomposable if and only if there exists an angle $\ga$ whose impression has
non-empty interior in $J(P)$; in this case the impressions of \emph{all} angles
coincide with $J(P)$.
\end{theorem}

From the topological standpoint, if the Julia set is indecomposable then one
cannot use impressions to further study its structure: representing $J(P)$ as
the union of smaller more primitive continua is impossible in this case.
Besides, if $J(P)$ is indecomposable then the results of this paper are
immediate. Thus, in the lemmas below we assume that no impression has interior
in $X$.

By $\ch(A)$ we
denote the convex hull of a planar set $A$. Lemma~\ref{jic} studies how ray
continua intersect. If two sets $A_1, A_2\subset \ucirc$ are such that $\ch(A_1)\cap
\ch(A_2)=\0$ then they are said to be \emph{unlinked}.

\begin{lemma}\label{jic} Suppose that $K_1, K_2$ are disjoint ray continua connected
to finite sets of angles $A_1, A_2$ respectively. Then $A_1$ and $A_2$ are unlinked.
\end{lemma}

\begin{proof}Since $K_1\cap K_2=\0$,  $A_1\cap A_2=\0$. If $A_1$ and $A_2$ are not unlinked,
then there exists $\al_1,\al_2\in A_1$ such that $A_2$ separates $\al_1$ and
$\al_2$. This clearly implies that $K_1\cap K_2\ne\0$, a contradiction.
\end{proof}

Recall that a continuum $X$ is \emph{connected im
kleinen at $x$} if for each open set $U$ containing $x$ there exists a
connected set $C\subset U$ with $x$ in the interior of $C$. A continuum $X$ is
\emph{locally connected} at $x\in X$ provided for each neighborhood $U$ of $x$
there exists a connected and open set $V$ such that $x\in V\subset U$. Observe
that sometimes different terminology is used (see the discussion in
\cite{bo06}).

Lemma~\ref{condklei} gives a sufficient condition for a continuum $X$ to be
connected im kleinen at some point $x$. It was proven in \cite{bo06}, Theorem
3.5, so we give it here without a proof. The idea is to establish ``short
connections'' among impressions which cut $x$ off the rest of $X$ and apply it
to prove that $X$ is connected im kleinen at some points. Recall that if $\imp(\ta)$ is
disjoint from all other impressions then we call $\ta$ a \emph{K-separate}
angle.

\begin{lemma}\label{condklei} Suppose that $\ta$ is a K-separate angle and
$\imp(\ta)=\{x\}$ is a singleton. Then arbitrarily close to $\ta$ there are
angles $s<\ta<t$ such that $\imp(s)\cap \imp(t)\ne \0$. Also, $J$ is connected
im kleinen at $x$.
\end{lemma}

On the other hand,  some conditions on $X$ imply that it is nowhere connected im kleinen,
or connected im kleinen at very few points. The next lemma combines Lemma 3.9
and Lemma 3.10 proven in \cite{bo06} and is given here without a proof.

\begin{lemma}\label{cikalt} The following holds.

\begin{enumerate}

\item If $X$ is connected im kleinen at $x$ then for any $\e$ there exists
$\ta$ such that $\imp(\ta)\subset B(x, \e)$ (in particular, there are angles
with impressions of arbitrarily small diameter).

\item Suppose that there exists $\delta>0$
such that for each $\ta\in \ucirc$, $\dia(\imp(\ta))>\delta$. Then $X$ is nowhere
connected im kleinen.

\item Suppose that the intersection $Z$ of all impressions is not empty.
Then the only case when $X$ is connected im kleinen at a point is (possibly)
when $Z=\{z\}$ is a singleton and $X$ is connected im kleinen at $z$.

\end{enumerate}

\end{lemma}

\section{Wandering continua for uniCremer polynomials}\label{wandecon}

In Section~\ref{wandecon} we use Thurston's invariant geometric laminations
defined in \cite{thur85}. A \emph{geometric lamination} is a compact set $\lam$
of chords in $\ol{\disk}$ and points in  $\ucirc$ such that any two distinct
chords can meet, at most, in an end-point (i.e., the intersection of any two
distinct chords is either empty or a point in $\ucirc$). We refer to a
non-degenerate chord $\ell\in \lam$ as a \emph{leaf} and to a point in $\lam$
as a \emph{degenerate leaf} (by ``leaves'' we mean non-degenerate leaves, and
by ``(degenerate) leaves'' we mean both types of leaves). A degenerate leaf may
be an endpoint of a leaf. If $\ell\cap \ucirc=\{a,b\}$ for a leaf $\ell$, we write
$\ell=ab$. We denote by $\lam^*$ the union of all leaves in $\lam$. Then
$\lam^*\cup \ucirc$ is a continuum. We can extend $\si:\ucirc\to \ucirc$ over
$\lam^*$ by mapping $\ell=ab$ linearly onto the chord $\si(a)\si(b)$ and
denoting this chord by $\si(\ell)$. A \emph{gap} $G$ is the closure of a
complementary domain of $\disk\setminus \lam^*$. A geometric lamination $\lam$
is \emph{$d$-invariant} if $\si$ preserves gaps and leaves of $\lam$ in the
following sense:

\begin{enumerate}

\item (Leaf invariance) For each leaf $\ell\in \lam$, $\si(\ell)$ is a
(degenerate) leaf in $\lam$ and there exist $d$ pairwise disjoint leaves
$\ell_1,\dots,\ell_d$ in $\lam$ such that for each $i$, $\si(\ell_i)=\ell$.

\item (Gap invariance) For each gap $G$ of $\lam$, $\si(\bd(G))$ is a
(degenerate) leaf or the boundary of a gap $G'$ of $\lam$. We denote by
$\si(G)$ the (degenerate) leaf or the gap $G'$, respectively. If
$\si(G)=G'$ is a gap then we also require that $\si|_{\bd(G)}:\bd(G)\to \bd(G')$
be the composition of a positively oriented monotone map and a positively oriented covering map.
\end{enumerate}

Wandering ray continua are defined in Section~\ref{intro}.
One can iterate the set of angles $\A$ connected to a wandering ray continuum;
by Lemma~\ref{jic} the sets $\si^n(\A)$ are unlinked.

Observe that it may happen that $K$ is contained in a larger wandering ray
continuum $K'$. However the growth of the set of angles connected to wandering
ray continua is limited as the following theorem shows (the result is due to Kiwi~\cite{kiwi02}, see
also \cite{bl02}).

\begin{theorem}\label{limwan}\cite{kiwi02} A set of angles such that all its images
under $\si$ are unlinked consists of no more than $2^d$ angles.
\end{theorem}

Theorem~\ref{lamexist} shows that wandering ray continua connected to a
non-trivial set of angles give rise to geometric laminations.

\begin{theorem} \label{lamexist} Suppose $P$ is a polynomial of degree $d$
with connected Julia set $J$ which contains a wandering ray continuum $K'$ and
$\A'$ is the set of angles connected to $K'$. Suppose that $|\A'|>1$. Then
there exists $N$ such that the grand orbit of the ray continuum $K=P^N(K')$ is
well-defined and gives rise to a $d$-invariant geometric lamination
$\lam(K')=\lam(K)$ such that $\A'$ is contained in a leaf or a gap of
$\lam(K')$.
\end{theorem}

\begin{proof}
First we show that any pullback $A$ of any forward iterate of $K'$ is
tree-like. Indeed, $A$ is wandering since so is $K'$. On the other hand, if $A$
is not tree-like then it contains the boundary of a Fatou domain, and by the
Sullivan theorem \cite{sul85} cannot be wandering. Hence all pullbacks of
images of $K'$ are tree-like. Then by a theorem of J. Heath \cite{hea96} the
map $P$ on a pullback $A$ is not one-to-one only if $A$ contains a critical
point of $P$.

Choose $K=P^N(K')$ so that the following holds. Since $K'$ is wandering then
for each critical point $c$ any image of $K'$ may contain only one iteration of
$c$. Choose a forward image $K$ of $K'$ so that for each critical point $c$
either some forward image of $c$ belongs to $K$, or the orbit of $c$ is
disjoint from the orbit of $K$. In particular, $K, P(K), \dots$ do not contain
critical points. Let us show that then there are $d^m$ pairwise disjoint
pullbacks of $P^n(K)$ of order $m\le n$ none of which contains a critical
point. Indeed, suppose that a pullback $A$ of $P^n(K)$ of order $m$ contains
a critical point $c$. Since by the choice of $K$ there exists $i>0$ such that
$P^i(c)\in K$ we then get that $P^n(K)$ contains both $P^{n+i}(c)$ and
$P^m(c)$, a contradiction. By \cite{hea96} this implies that all powers of $P$
restricted onto a pullback $A$ of $P^n(K)$ of order $m\le n$ are one-to-one.

This implies that the pullbacks of images of $K$ are in fact well-defined as
sets: any two pullbacks $A, B$ are either the same or disjoint. Indeed, suppose
that $A$ is a pullback of $P^n(K)$ of order $m$ and $B$ is a pullback of
$P^r(K)$ of order $s$. Suppose that $A\cap B\ne \0$ and show that then $A=B$.
Choose a point $x\in A\cap B$ and consider several cases. For definiteness
suppose that $m>s$. First let us show that $n-m=r-s$. Indeed, $P^m(x)\in
P^n(K)$ and $P^s(x)\in P^r(K)$. Then the latter implies that
$P^m(x)=P^{m-s}(P^s(x))\in P^{m-s+r}(K)$. Hence $P^n(K)$ and $P^{m-s+r}(K)$ are
not disjoint (both contain $P^m(x)$) and hence $n=m-s+r$ because $K$ is
wandering. Now, by the above there is only one pullback of $P^m(K)$ containing
$P^s(x)$, namely $P^r(K)$. Hence $P^s(A)=P^r(K)$ and  $A$ is the pullback of
$P^r(K)$ of order $s$ along the orbit $x, P(x), P^s(x)$. Since the same holds
for $B$ we conclude that $A=B$.

Choose a maximal set of angles $\A$ connected to $K$ and containing
$\si^N(\A')$ (clearly, all angles from $\si^N(\A')$ are connected to $P^N(K)$).
Denote a pullback of $P^n(K)$ by $P$ of order $m$ by $K(m, n, i)$ where
different numbers $i$ correspond to different pullbacks of the same order $m$
of the same image $P^n(K)$ of $K$. In other words, $K(m, n, i)$ is the $i$-th
component of the set $P^{-m}(P^n(K))$. To $K(m, n, i)$ we associate the set of
angles $\Ta(m,n,i)$ which are all the angles from $\si^{-m}(\si^n(\A))$
connected to $K(m, n, i)$. The fact that all sets $K(m,n,i)$ in the grand orbit
of $K$ are pairwise disjoint and the definition of the set of angles connected
to a ray continuum imply that the sets $\Ta(m,n,i)$ have the same property: any
two such sets of angles either coincide or are disjoint.

If $G(m,n,i)=\ch(\Ta(m,n,i))$, then by Lemma~\ref{jic} it follows that all the sets $G(m,n,i)$
are pairwise disjoint. By Theorem~\ref{limwan} all the sets $\Theta(m,n,i)$ are
finite. Hence all $G(m,n,i)$ are pairwise disjoint finite gaps,  leaves and
points mapped onto each other by $\si$ and its powers. Moreover, $\si$
restricted on the sets $G(m, n, i)$ satisfies all the properties described in
the definition of the lamination because this corresponds to the action of the
map $P$ on the plane. Hence $\cup_{m,n,i} G(m,n,i)$ is a $\sigma$-invariant non
compact lamination $\lam(G)$. It follows easily that the closure of this
non-compact lamination is a $\sigma$-invariant lamination $\lam(K)$. Observe
that by the construction $\A'$ is contained in a leaf or a gap, and hence
$\lam(K)$ is non-trivial (because the cardinality of $\A$ is more than $1$).
\end{proof}

We use geometric laminations in the proof of the following theorem.

\begin{theorem}\label{nowander}
Suppose that $J$ is the Julia set of a basic uniCremer polynomial $P$ of degree
$d$. Suppose that $K'\subset J$ is a wandering ray continuum connected to a set
of angles $\A'$ with $|\A'|\ge 2$. Then there exists an $n$
such that $|\si^n(\A')|$ is a singleton and for every $\ta\in \A'$, $\cup_{n\ge 0}
\si^n(\ta)$ is not dense in $\ucirc$. In particular, $K'$ is pre-critical, and if
a point $x\in J$ is bi-accessible then it is either precritical or preCremer.
\end{theorem}

\begin{proof}
By Theorem~\ref{lamexist}, there exists a $d$-invariant lamination
$\lam(K')$ such that $\A'$ is contained in a gap or leaf of $\lam(K')$. Let $p$
be the fixed Cremer point of $P$.  We can replace $K'$ by $K$ constructed in
Theorem~\ref{lamexist} and replace $\A'$ by the set of all angles $\A$ connected to
$K$. It suffices to prove the theorem for $K$ and $\A$. By way of contradiction
suppose that for some $m>1$ and every $n$ we have $|\si^n(\A)|=m$.
Denote the grand orbit of $K$ under $P$ by $\Ga(K)$; denote the grand orbit of
$\A$ under $\si$ (whose closure is the geometric lamination $\lam(K)$) by $\Ta(\A)$.
It is easy to see that there are no critical leaves on the boundaries of sets from $\Ta(K)$
(a set $T\in \Ga(K)$ may cover a critical point of $P$, but then the corresponding set from
$\Ta(\A)$ will be mapped onto its image in a $k$-to-$1$ fashion so that one can draw
critical chords inside its convex hull, but not on its boundary).

The idea of the proof is to show that in the situation of the theorem one can always find
a continuum $X\subset J$ to which Corollary~\ref{degimp1} applies leading to a contradiction
with the definition of a uniCremer polynomial. To do so, we analyze the
dynamical behavior of $K$ and show that at some moment an image of $K$ will
be mapped ``farther away'' from $p$ by the appropriate power of $P$. To implement this plan
we need the following construction.

For each element $D$ of $\Ga(K)$, the union of
the rays of angles from the appropriate element of $\Ta(\A)$ and $D$ cut the plane
into open wedges one of which contains $p$. Denote by $W_D$ the  union of this  wedge,
its two boundary rays and $D$ (clearly, $W_D$ is closed). Simultaneously
consider the corresponding leaf $\ell_D$ on the boundary of the gap/leaf of $\Ta(K)$
corresponding to $D$ and  the closure $H_D$ of the component of
$\disk\sm \ell_D$ corresponding to angles whose rays are contained in $W_D$.
Then consider the intersection $W$ of sets $W_D$ and the intersection $H$ of
sets $T_D$ over all elements $D$ in the grand orbit of $K$.

By the construction
$H$ is a convex subset of $\disk$ whose boundary consists of leaves of $\lam(K)$ and a nowhere
dense subset $H\cap \uc$ of $\uc$. Since by the construction
no convex hull of a set of $\Ta(\A)$ crosses $H$, either $H$ is a point in $\uc$, or
$H$ is a gap or a leaf of $\lam(K)$. Denote by $A(p)$ the set of all angles whose impressions contain
$p$. Then it follows that $A(p)\subset H\cap \uc=H'$. Clearly, $\si(A(p))\subset A(p)$
and hence $\si(H')\cap H'\supset \si(A(p))$.

Consider first the case when $H=H'$ is a point. Then by the above $H=H'=A(p)$ is a
fixed point of $\si$. Choose an element $T$ of $\Ta(\A)$ close to $H$ and consider
the corresponding element $Q$ of $\Ga(K)$. Then there is
a unique open wedge $N$ - one of the components of $\Complex\sm [Q\cup \bigcup_{\al\in T}R_\al]$ - into
which $Q$ is mapped by $P$. Consider the union $X$ of $Q$ with the intersection of $J$
and $N$. It is easy to check that Corollary~\ref{degimp1} applies to $X$ (with $N$ being its unique
exit continuum). This implies that $X$ contains a $P$-fixed point $a$ with several  rays landing at $a$
which rotate under $P$ (i.e., $a$ is a bi-accessible fixed point),
a contradiction with the definition of basic uniCremer polynomials.

Suppose next that $H$ is a gap or a leaf. Observe that there are no
critical leaves on the boundary of $H$. Indeed, if $\ell\in \partial(H)$ is
a critical leaf then it cannot belong to an element from $\Ta(\A)$, hence it is a limit of such elements.
By the properties of laminations this is impossible in the case when $\si(H)$ is not a point
because then the images of these elements of $\Ta(\A)$ will cross the image of $H$. Now, if $\si(H)$
is a point this point must coincide with $A(p)$ and is fixed. In this case the images of elements of
$\Ta(\A)$ which approximate $\ell$ will cross $H$ itself, again a contradiction.

Let us show (by way of contradiction) that $\si(H')\subset H'$.
By the above $\si(H')$ cannot be a point. Then the fact that $\si(H')\not\subset H'$
implies that $H'$ and $\si(H')$ have one or two points in common. Suppose they
have only one point $a$ in common. Then $a\in \si(A(p))$ and $\si(a)=a$.
For geometric reasons then there is a leaf $\ell\in \partial H$ and there is a leaf
$\ell'\in \partial H\sm \partial (\si(H))$ which have the point $a$ in common.
Elements of $\Ta(\A)$ are disjoint from these leaves because they are wandering;
on the other hand, by the construction of $\lam(K)$ elements of $\Ta(\A)$
must converge to $\ell$ and $\ell'$.
Since for each set $D\in \Ga(K)$ we have $a\in H_D$ (because
$p\in \imp(a)$), hence by the construction \emph{both} leaves $\ell, \ell'$
must belong to $H$, a contradiction.

Suppose that $\si(H')$ and $H'$ have two points in common. Since by the assumption
$\si(H')\not\subset H'$ this means that both $H$ and $\si(H)$ are gaps with a common boundary
leaf $\ell$. Thus, $\ell$ is isolated in $\lam(K)$ and lies on the boundary
of a wandering set from $\Ta(\A)$. On the other hand,
$\si(H)\cap H\supset \si(A(p))\si^2(A(p))$ which means that no endpoint of $\ell$ can
be wandering, a contradiction. Thus, $\si(H')\subset H'$.

Since by the above there are no critical leaves in $\partial H$, there must exist a periodic
leaf $\ell\in \partial H$. This is clear is $H$ is a leaf. Otherwise,
leaves from $\partial H$ map onto one another
and never collapse. Moreover, for each $\ell\in \partial H$ let $M_\ell$ be the arc
with the same endpoints as $\ell$ and disjoint from $H$. Then for distinct leaves from
$\partial H$ their arcs are disjoint, and if the length of the arc
$M_{\ell}$ is smaller than $\frac 1{d+1}$ then the length of the arc $M_{\si(\ell)}$
is greater than that of $M_{\ell}$. Since there cannot be infinitely many such
arcs with bounded from below length, the exists a periodic leaf $\ell\in \partial H$.
Suppose that its period is $m$.

Choose an element $T$ of $\Ta(\A)$ close to $\ell$ and consider
the corresponding element $Q$ of $\Ga(K)$. Then there is
a unique open wedge $N$ - one of the components of $\Complex\sm [Q\cup \bigcup_{\al\in T}R_\al]$ - into
which $Q$ is mapped by $P^m$. Consider the union $X$ of $Q$ with the intersection of $J$
and $N$. It is easy to check that Corollary~\ref{degimp1} applies to $X$ (with $N$ being its unique
exit continuum) and $P^m$. This implies that $X$ contains a $P^m$-fixed point $a$ with several landing rays
which rotate under $P$, a contradiction with the definition of basic uniCremer polynomials.

Finally, if the $\si$-orbit of $\A$ is dense in $\uc$ then clearly we will be able to find such that $m>n$ that
$\si^m(K))$ is separated from $p$ by $K$ united with the rays of the angles from $\A$. Then as before
by Corollary~\ref{degimp1} we get a contradiction with the definition of a uniCremer polynomial.
\end{proof}

Observe that even if the set $\si^n(\A)$ is a point for some $n$, the construction of the sets $W$ and
$H$ goes through and shows that in the basic uniCremer case the dynamical behavior of $K$
(and $\A$) is very specific.

\section{Main results}\label{quadra}

First we establish some facts which may be of independent interest and serve as
an additional motivation for us. As was explained in the Introduction, Kiwi's
results \cite{kiwi97} do not apply to uniCremer polynomials. Still, one could
hope to model (topologically) a basic uniCremer Julia set $J$ by monotonically
mapping $J$ onto a locally connected continuum. It turns out that this is
impossible. In the quadratic case we proved in \cite{bo06b} that a monotone map
of $J$ onto a locally connected continuum collapses $J$ to a point. However the
proofs in \cite{bo06b} rely also upon results of \cite{grismayeover99} not
known for basic uniCremer polynomials. Using a new argument we fill this gap
and prove Theorem~\ref{nomonproj}.

Let us state some results of \cite{bo06b}. An \emph{unshielded} continuum
$K\subset \C$ is a continuum which coincides with the boundary of the unbounded
 component of $\Complex\sm K$. Given an external (conformal) ray $R$, a
crosscut $C$ is said to be \emph{$R$-transversal} if $\ol{C}$ intersects
$\ol{R}$ only once and the intersection is transverse and contained in $C\cap
R$; if $t\in R$ then by $C_t$ we always denote a $R$-transversal crosscut  such
that $C_t\cap \ol{R}=\{t\}$. Given an external ray $R$ we define   the
\emph{(induced) order} on $R$ so that $x<_R y\,(x, y\in R)$ if and only if the
point $x$ is ``closer to $K$ \textbf{on the ray $R$} than $y$''. Given an
external ray $R$, we call a family of $R$-transversal crosscuts $C_t$, $t\in R$
an \emph{$R$-defining} family of crosscuts if for each $t\in R$ there exists a
$R$-transversal crosscut $C_t$ such that $\dia(C_t)\to 0$ as $t\to K$ and
$\sh(C_t)\subset \sh(C_s)$ if $t<_R s$.

\begin{lemma}\cite[Lemma 2.1]{bo06b}\label{dist0} Let $K$ be an
unshielded continuum and $R$ be an external ray to $K$. Then there exists an
$R$-defining family of $R$-transversal crosscuts $C_t$, $t\in R$.
\end{lemma}

Now we can prove Theorem~\ref{nomonproj}.

\begin{theorem}\label{nomonproj} Suppose that $P$ is a basic uniCremer polynomial
and $\ph:J\to A$ is a monotone map of $J$ onto a locally connected continuum
$A$ or a point. Then $A$ is a singleton.
\end{theorem}

\begin{proof} By way of contradiction suppose that $\ph:J\to A$ is a monotone map
onto a locally connected non-degenerate continuum $A$. Define the map $\Phi$ on
the complex plane $\C$ so that it identifies precisely \emph{fibers}
(point-preimages) of $\ph$ and does not identify any points outside $J$. Since
the decomposition of $\C$ into fibers of $\ph$ and points of $\C\setminus J$ is
upper-semicontinuous, the map $\Phi$ is continuous. Since $J$ and hence all its
subcontinua are non-separating then by the Moore Theorem \cite{m62} the map
$\Phi$ maps $\C$ onto $\C$, and so $\Phi(J)=\ph(J)=A$ is a \emph{dendrite}
(locally connected continuum containing no simple closed curve). External
(conformal) rays $R_\al$ in the $J$-plane are then mapped into continuous
pairwise disjoint curves $\Phi(R_\al)$ in the $A$-plane called below
\emph{$A$-rays}. Clearly, if $R_\al=R$ lands then so does $\Phi(R)$. Let us
show that $\Phi(R)$ \emph{always} lands, and in fact the impression $\imp(\al)$
maps under $\ph$ to the landing point of the $A$-ray $\Phi(R)$. By
Lemma~\ref{dist0} there exists an $R$-defining family of crosscuts $C_t$. Since
$\Phi$ is continuous then $\dia(\Phi(C_t))\to 0$ as $t\to J$. Consider two
cases.

Suppose that $\Phi(C_t)$ is an arc for all $t\in R$, and hence a crosscut of
$A$. Since $A$ is locally connected then by Carath\'eodory theory $\Phi(C_t)$
converges to a unique point $x\in A$ which implies that $\Phi(R)$ lands. Also,
by Carath\'eodory theory the $\Phi$-images of the closures of the shadows
$\sh(C_t)$ converge to the same point $x$ which belongs to them all. Since the
intersection of the closures of the shadows $\sh(C_t)$ is the impression
$\imp(\al)$ of $\al$ then $\ph(\imp(\al))=\{x\}$ as desired. Otherwise there
exists $t\in R$ such that $\Phi(C_t)$ is a simple closed curve. This can only
happen if the endpoints of $C_t$ map under $\ph$ to the same point, say, $z$.
It follows that $\Phi(C_s)$ is a simple closed curve for all $s<_R t$, and
these curves all contain $z$. Thus, $\Phi(R)$ lands at $z$. Moreover, in this
case the shadows $\sh(C_s)$ are mapped inside the simple closed curves
$\Phi(C_s), s<_R t$ which as before implies that $\ph(\imp(\al))=\{z\}$.

Since impressions are upper-semicontinuous, the family of $A$-rays is
continuous, and the map $\psi$ associating to every angle $\al$ the landing
point of the $A$-ray $\Phi(R_\al)$ is a continuous map of $\ucirc$ onto $A$.
Define the \emph{valence} of a point $y\in A$ as the number of components of
the set $A\setminus \{y\}$ if it is finite and infinity otherwise. Let
$\mathcal B_y$ be the set of $A$-rays landing at $y$. Then the number of
components of $\ucirc\setminus \mathcal B_y$ equals the valence of $y$. Indeed,
if $(\al, \be)$ is a component of $\ucirc\setminus \mathcal B_y$ then the
$A$-rays of angles in $(\al, \be)$ land in a component of $A\setminus \{y\}$
and for distinct components of $\ucirc\setminus \mathcal B_y$ we get distinct
components of $A\setminus \{y\}$. In fact there is exactly one component of
$A\setminus \{y\}$ contained in the appropriate wedge in the plane formed by
the $A$-rays $\Phi(R_\al)$ and $\Phi(R_\be)$. Indeed, otherwise choose angles
$\ga, \ta\in (\al, \be)$ so that their $A$-rays $\Phi(R_\ga)$ and $\Phi(R_\ta)$
land at points $x, z$ from distinct components of $A\setminus \{y\}$. Then the
path $\psi([\ga, \ta])$ connects points $x$ and $z$ inside $A$ and hence must
pass through $y$. On the other hand by the construction $y\nin \psi([\ga,
\ta])$, a contradiction.

We show that except for a countable set of points there are no more than two
$A$-rays landing at $y\in A$. Let $Q'$ be the set of all points $y\in A$ with
finite valence for which there are infinitely many $A$-rays landing at $y$.
Then by the previous paragraph $\mathcal B_y$ has a non-empty interior for any
point $y\in Q'$, and so $Q'$ is countable. On the other hand, by Theorem 10.23
of \cite{nad92} the set $Q''$ of all branch points of $A$ is countable (a
\emph{branch point} is a point of valence greater than $2$). Hence, for any
point $y\in A\setminus (Q'\cup Q'')$ such that its valence is greater than $1$
(such points are called \emph{cutpoints}) exactly two $A$-rays land at $y$.

Choose a point $y\in A\setminus (Q'\cup Q'')$ using a bit of dynamics. Let $H$
be the union of grand orbits of $p$ and all critical points of $P$ and set
$Q'''=\ph(H)$. Then $\hat Q=Q'\cup Q''\cup Q'''$ is countable. Choose $y\in
A\setminus \hat Q$ to be a cutpoint of $A$. Then exactly two $A$-rays land at
$y$. Moreover, let $\ph^{-1}(y)=K$; then by the choice of $y$ forward images of
$K$ avoid $p$ and critical points of $P$. Since impressions are mapped by $\ph$
into points, there are exactly two angles $\al, \be$ with $K=\imp(\al)\cup
\imp(\be)$ and the impressions of other angles are disjoint from $K$.

By Theorem~\ref{nowander} there are integers $0\le l<m$ such that $P^l(K)$ and
$P^m(K)$ intersect. Since forward images of $K$ avoid critical points,
$\si^r(\al)\ne \si^r(\be)$ for any $r$ and hence both
$P^l(K)=\imp(\si^l(\al))\cup \imp(\si^l(\be))$ and $P^m(K)=\imp(\si^m(\al))\cup
\imp(\si^m(\be))$ are unions of impressions of two distinct angles. Now, if the
pair of angles $\si^l(\al), \si^l(\be)$ maps into itself by $\si^{m-l}$ then
$P^l(K)$  is a $P^{m-l}$-fixed continuum not containing $p$, hence by
Corollary~\ref{degimp1} it is a repelling periodic point at which two rays land, a
contradiction. Hence there exists an angle $\ga\in \{\si^m(\al),
\si^m(\be)\}\setminus \{\si^l(\al), \si^l(\be)\}$ such that $\imp(\ga)$
non-disjoint from the $P^l(K)$. If we now pull $P^l(K)$ back to $K$ we will get
an angle $\ga'\nin \{\al, \be\}$ whose impression is non-disjoint from $K$, a
contradiction.
\end{proof}

Since we assume that $J$ is decomposable, all impressions are proper and
have empty interior. In particular no countable union of impressions coincides
with $J$. We begin by studying \emph{red dwarf} Julia sets, i.e. uniCremer
Julia sets such that impressions of all angles contain the Cremer point $p$.

\begin{lemma}\label{redwarf} If $J$ is a red dwarf Julia set then {\rm (1)} the
intersection $K$ of all impressions contains all forward images of all critical
points, {\rm (2)} there exists $\e>0$ such that the diameter of any impression is greater
than $\e$, {\rm (3)} there are no points at which $J$ is connected im kleinen,
and {\rm (4)} no point of $J$ is biaccessible and $p$ is not
accessible from $\C\setminus J$.
\end{lemma}

\begin{proof} Let us show that for any angle $\al$ its impression
$\imp(\al)$ contains all critical images. Indeed, otherwise there exists a
critical point $c$ such that $P(c)\nin \imp(\al)$ which implies that $c\nin K$.
Choose a curve $T$ starting at $P(c)$, going to infinity and bypassing
$\imp(\al)$. Then the pullback of $T$ containing $c$ cuts the plane into at
least two components each of which contains at least one pullback of
$\imp(\al)$, a contradiction. Hence $K$ contains all critical values, and since
$K$ is forward invariant, $K$ contains all forward images of all critical
points.

Clearly, (1) implies (2) (by \cite{mane93} there is a recurrent critical point
whose forward orbit avoids $p$). By Lemma~\ref{cikalt}(2) $J$ is nowhere
connected im kleinen. Moreover, no point $x\in J$ is biaccessible from
$\C\setminus J$. Indeed, if $x\ne p$ is biaccessible then one of the two
half-planes into which $x$ and rays landing at $x$ cut the plane will not
contain $p$, hence all rays contained in that half-plane will not contain $p$
in their impressions, a contradiction. On the other hand, if $x=p$ we can take
a preimage of $p$ and get the same contradiction. Hence $p$ cannot be
accessible from $\C\setminus J$ because if it is then the corresponding ray is
not fixed (periodic rays cannot land on $p$ by Douady and Hubbard
\cite{douahubb85}) and hence this ray and its image show that $p$ is
biaccessible, a contradiction.
\end{proof}

Lemma~\ref{imp} follows from Corollary~\ref{degimp1}.

\begin{lemma}\label{imp} If $K\subset J$ is a $P^n$-invariant continuum or singleton
not containing $p$ then $K$ is a singleton. In particular, if $\ga$ is a periodic
angle and $p\nin \imp(\ga)$ then $\imp(\ga)$ is a singleton.
\end{lemma}

\begin{proof}
All $P^n$-fixed points in $K$ are repelling, and by the definition
of a basic uniCremer polynomial at each of them exactly one $P^n$-fixed ray
lands. Hence by Corollary~\ref{degimp1} $K$ is a singleton.
\end{proof}

For $x\in J$ let $A(x)$ be the set of all angles whose impressions contain $x$,
and let $B(x)$ be the union of these impressions. Then $A(x)$ and $B(x)$ are
closed sets. In Lemma~\ref{pnotksep} we study these sets for a periodic $x$.

\begin{lemma}\label{pnotksep} If $x$ is periodic then one of the following holds:

\begin{enumerate}

\item $x\nin B(p)$, then $\{x\}=B(x)=\imp(\ta)$ for a periodic K-separate angle
$\ta$, $A(x)=\{\ta\}$, and $x$ is a repelling periodic point;

\item $x\in B(p)$, then $B(x)$ is non-degenerate and no angle $\ta\in A(x)$ is K-separate.

\end{enumerate}

In particular, no angle $\ta\in A(p)$ is K-separate, and $B(p)$ is
non-degenerate.

\end{lemma}

\begin{proof} Consider first the case when $x=p$. We need to show that then
the case (2) holds. First let us show that $B(p)$ is non-degenerate. Suppose
that $A(p)$ is infinite. Since $A(p)$ is invariant it follows from a well-known
result from the topological dynamics of locally expanding maps that
$\si|_{A(p)}$ is not one-to-one. Hence by \cite{hea96} $B(p)$ contains a
critical point and cannot coincide with $p$. Suppose that $A(p)$ is finite.
Then $A(p)$ contains periodic angles, hence $B(p)$ contains periodic points
distinct from $p$ and hence again $B(p)$ is not degenerate.

Let us show that no angle $\ta\in A(p)$ is K-separate. We may assume that
$A(p)=\{\ta\}$ consists of just one angle and need to show that $\ta$ is not
K-separate. Clearly, $\si(\ta)=\ta$. Denote the landing point of $R_\ta$ by
$x$, and show that there is a critical point $c\in B(p)=\imp(\ta)$. Indeed,
otherwise choose a neighborhood $U$ of $B(p)$ such that no critical points
belong to $\ol{U}$, consider the set of all points never exiting $\ol{U}$, and
then the component $K$ of this set containing $p$. Such sets are called
\emph{hedgehogs} (see papers by Perez-Marco \cite{pere94} and \cite{pere97})
and have a lot of important properties. In particular, by \cite{pere94} and
\cite{pere97} $K$ cannot contain a periodic point other than $p$, a
contradiction (clearly, $B(p)\subset K$ and $B(p)$ contains $x$). Hence $c\in
B(p)$ for some critical point $c$. This implies that for some integer $i, 1\le
i\le d-1$ we have $c\in \imp(\ta+i/d)$ and hence $B(p)=\imp(\ta)$ is not
disjoint from the impression of another angle and $\ta$ is not K-separate as
desired.

Suppose now that $x\nin B(p)$ is a periodic point of period $m$. Then $p\nin
B(x)$ and hence by Lemma~\ref{imp} $B(x)$ is degenerate. By the above quoted
topological result this implies that $A(x)$ is finite and therefore, by the
assumptions on $P$, $A(x)=\{\ta\}$ where $\ta$ is periodic and K-separate. Now
assume that $x\in B(p)$. Then there is an angle $\al\in A(p)$ with $p, x\in
\imp(\al)$ and hence $p\in B(x)$ and $B(x)$ is not degenerate.
Now, if there are more than one angle in $A(x)$ then all such angles are
not K-separate and we are done. If however there is only one angle in $A(x)$
then this angle is $\al=\ta$ which belongs to $A(p)$ and by the previous
paragraph is not K-separate either.
\end{proof}

Lemma~\ref{pnotksep} implies a few facts: e.g., if $\ta$ is a K-separate
periodic angle then $\imp(\ta)$ is a repelling periodic point.
Lemma~\ref{exishort} shows that in some cases there are lots of such angles.
For $F\subset \ucirc$ denote by $\imp(F)$ the set $\cup_{\ta\in F}\imp(\ta)$.
Let $E$ be the set of all K-separate periodic angles; by Lemma~\ref{pnotksep}
each angle in $E$ has degenerate impression.

\begin{lemma}\label{exishort} Suppose $p\nin \imp(\ta)$ for some angle $\ta$.
Then {\rm (1)} $B(p)$ is a nowhere dense subset of $J$, {\rm (2)} the set $E$
is dense in $\ucirc$, {\rm (3)} the set $\imp(E)$ is dense in $J$, {\rm (4)}
for a closed set of angles $F\ne \ucirc$ the set $\imp(F)$ is a proper subset
of $J$, {\rm (5)} $J$ is connected im kleinen at every point $y\in \imp(E)$,
{\rm (6)} in any arc $W\subset \ucirc$ there are two angles whose impressions
meet.
\end{lemma}

\begin{proof} To prove the first claim it is enough to show that $B(p)\ne J$.
Suppose otherwise. Since by our standing assumption $J$ is decomposable then no
finite union of impressions can coincide with $J$, and $A(p)$ is infinite.
Then in any open arc $V$ there are angles whose impressions meet. Indeed, we
can find a big integer $N$ such that $\si^{-N}(A(p))\cap V$ is infinite. Each
angle from $\si^{-N}(A(p))$ contains a $P^N$-preimage of $p$ in its impression.
Since there are only finitely many $P^N$-preimages of $p$ then there are two
angles $\ga, \ga'\in \si^{-N}(A(p))\cap V$ which contain the same preimage of
$p$ and therefore meet as desired.

Denote by $U$ an open arc containing $\ta$ such that for any angle in $U$ its
impression does not contain $p$ (then by Lemma~\ref{imp} periodic angles from
$U$ have degenerate impressions). Assume that $(\ga, \ga')\subset U$ is an arc
such that the impressions of $\ga, \ga'$ meet. The union $Z=R(\ga)\cup
\imp(\ga)\cup \imp(\ga')\cup R(\ga')$ cuts the plane into two half-planes $H$
and $G$; assume for the sake of definiteness that $p\in G$. Since no finite
union of impressions coincides with $J$ then the union of impressions of
angles from $(\ga, \ga')$ is not contained in $Z$. Hence there exists a
point $h\in J\cap H$. Then the impression of an angle from $(\ga', \ga)$ cannot
contain $h$ because $Z$ separates this angle's ray from $h$. On the other hand,
the impression of an angle from $[\ga, \ga']$ cannot contain $p$ either. Hence
no impression contains $h$ and $p$ simultaneously, implying that $B(p)\ne J$.

By Lemma~\ref{pnotksep} $E$ is the set of all periodic angles such that landing
points of their rays do not belong to $B(p)$. By the upper semi-continuity of
impressions, $E$ is open in the set of all periodic angles. Since $E$ is
invariant, $E$ is dense in $\ucirc$ which proves (2). We claim that $\imp(E)$
is dense in $J$. Indeed, otherwise there exists an open set $U\subset J$
disjoint from $\imp(E)$. Since periodic points are dense in $J$, they are dense
in $U$, and by the definition of $E$ all these periodic points belong to
$B(p)$. Hence $B(p)$ has non-empty interior, a contradiction with (1). Thus,
$\imp(E)$ is dense in $J$ which proves (3). The claim (2) implies (4). By
Lemma~\ref{condklei} the rest of the lemma follows.
\end{proof}

So far the results of this section use mainly topological tools. This changes
in the lemmas below where we rely upon both continuum theory and dynamics. Our
aim is to prove that the angles with dense in $\ucirc$ orbits have degenerate
impressions. Problems of this kind are often related to the dynamics of
critical points. The result  obtained in Lemma~\ref{densenotcrit}
enables us to apply some standard tools and seems to be interesting by itself.
However first we need a simple lemma.

\begin{lemma}\label{densequiv} Suppose that there exists an angle $\ta$ whose
impression does not contain $p$. Then the following statements are equivalent.

\begin{enumerate}

\item An angle $\al$ has a dense in $\ucirc$ orbit.

\item Any point of $\imp(\al)$ has a dense in $J$ orbit.

\item There exists a point in $\imp(\al)$ which has a dense in $J$ orbit.

\end{enumerate}

\end{lemma}

\begin{proof} Suppose that $\al$ has a dense in $\ucirc$ orbit and choose
$x\in \imp(\al)$. Let $y\in J$ and $\e>0$. By Lemma~\ref{exishort} there exists
$\gamma\in \ucirc$ such that $\imp(\gamma)=\{z\}$ and $d(z,y)<\e/2$. Since the
orbit of $\al$ is dense in $\ucirc$ and impressions are upper semi-continuous,
there exists $n>0$ such that $\si^n(\al)$ is so close to $\ga$ that
$\imp(\si^n(\al))\subset B(z,\e/2)$, and hence $d(P^n(x),y)<\e$. Therefore,
$\omega(x)=J$ and (1) implies (2). Clearly, (2) implies (3). Now, suppose that
(3) holds. Then $\imp(\omega(\al))=J$ which by Lemma~\ref{exishort}(4) implies
that $\omega(\al)=\ucirc$ as desired.
\end{proof}

Now we show that if there exists an angle $\ta$ whose impression
does not contain $p$ then no critical point has a dense orbit in $J$.

\begin{lemma}\label{densenotcrit} Suppose that there exists an angle $\ta$
whose impression does not contain $p$. Suppose that $\al$ is an angle with a
dense orbit in $\ucirc$. Then $\al$ is K-separate and $\imp(\al)$ does not contain
a critical point of $P$. In particular, no critical point can have a dense
orbit in $J$.
\end{lemma}

\begin{proof} The proof consists of several steps. Suppose that there are
finitely many angles $\al=\al_0, \al_1, \dots, \al_k$ such that the union $I$
of their impressions is a continuum. Let us show that then $I$ has to be
wandering (all its iterates are pairwise disjoint). By way of contradiction and
without loss of generality we may assume that $P(I)\cap I\ne \0$. Consider the
set $Q$ of all angles whose impressions are not disjoint from $B(p)$ and the
set $\imp(Q)$. Then $Q\ne \ucirc$ by Lemma~\ref{pnotksep}. Hence by
Lemma~\ref{exishort}(4) $\imp(Q)\ne J$. Let us show that we may assume that
\emph{all} images of $I$ intersect $\imp(Q)$.

By Lemma~\ref{pnotksep} the set $Q$ contains two angles with non-disjoint
impressions. Denote these angles $\ga$ and $\be$. Then the rays of these angles
together with their impressions cut the plane into two components $H$ and $G$
so that the periodic K-separate angles from $(\ga, \be)$ have point-impressions
in $H$ and the periodic K-separate angles from $(\be, \ga)$ have
point-impressions in $G$. Choose periodic K-separate angles $\ta'\in (\ga,
\be)$ and $\ta''\in (\be, \ga)$, and then choose $k<l$ so that $\si^k(\al)\in
(\ga, \be)$ is very close to $\ta'$ and $\si^l(\al)\in (\be, \ga)$ is very
close to $\ta''$. Then the continuum $Z=P^k(I)\cup P^{k+1}(I)\cup \dots \cup
P^l(I)$ connects points from $H$ to points from $G$. Hence by Lemma~\ref{jic}
$Z\cap \imp(Q)\ne \0$. Thus, from some time on the images of $I$ are
non-disjoint from $\imp(Q)$, and we may assume that in fact $I\cap \imp(Q)\ne
\0$ and hence \emph{all} images of $I$ intersect $\imp(Q)$.

Just like the set $B(p)$ was constructed as the union of all impressions
non-disjoint from $p$, and the set $\imp(Q)$ was constructed as the union of
all the impressions non-disjoint from $B(p)$, this process can be continued for
$k+1$ more steps resulting into a union of impressions which we will denote by
$T$. Clearly, $T$ is a closed; moreover, since no point of $\imp(E)$ can belong
to $T$ then $T$ is a proper subset of $J$. Since $\al$ has a dense orbit, by
Lemma~\ref{densequiv} the orbit of any point of $\imp(\al)$ is dense in $J$. On
the other hand, by the previous paragraph the orbit of $\imp(\al)$ is contained
in $T$ and $T\ne J$, a contradiction. So, the assumption that $I$ is not
wandering leads to a contradiction which implies that $I$ is wandering. Observe
that then by Theorem~\ref{limwan} $k\le 2^d$.

Let us now show that $\al$ is K-separate. Indeed, otherwise let $\al'\ne \al$
be such that $\imp(\al)\cap \imp(\al')\ne \0$. Then by the above the maximal
finite collection of angles $\al_0=\al, \al_1=\al', \dots, \al_k$ such that the
union $I$ of their impressions is connected consists of $k+1>1$ angles and is
such that $I$ is wandering. By maximality the impressions of other angles are
disjoint from $I$. However the orbit of $\al$ is dense which contradicts
Theorem~\ref{nowander}. Hence $\al$ is K-separate as desired. To complete the
proof it remains to notice that if a critical point $c$ belongs to $\imp(\al)$
then, because locally around $c$ the map $P$ is not one-to-one, there exists an
angle $\al'$ such that $\si(\al)=\si(\al')$ and $c\in \imp(\al')$ implying that
$\al$ is not K-separate. Hence $\imp(\al)$ does not contain critical points.
Since this holds for any angle with dense orbit we conclude that by
Lemma~\ref{densequiv} no critical point can have a dense orbit in $J$.
\end{proof}

Lemma~\ref{transimp} completes a series of claims made in Lemma~\ref{densequiv}
and Lemma~\ref{densenotcrit}.

\begin{lemma}\label{transimp} Suppose that there exists an angle $\ta$ whose
impression does not contain $p$. Let $\al$ be an angle such that the
$\si$-orbit of $\al$ is dense in $\ucirc$. Then $\imp(\al)$ is a singleton, and,
moreover, the angle $\al$ is K-separate.
\end{lemma}

\begin{proof} By Ma\~n\'e \cite{mane93} the closure $B'$ of the union of orbits of
all critical points of $P$ contains $p$. By Lemma~\ref{densenotcrit} $B'$ is
nowhere dense in $J$. By Lemma~\ref{exishort} there is an angle $\ga\in E$ such
that the singleton $\imp(\ga)$ is not contained in $B'$. By the upper
semi-continuity of impressions we can find an arc $U$ around $\ga$ so that the
union of impressions of angles from $\ol{U}$ is positively distant from $B'$.
By Lemma~\ref{condklei} we can find two angles $\tau'<\ga<\tau''$ in $U$ (the
order is considered within $U$) such that $\imp(\tau')\cap \imp(\tau'')\ne \0$.
Set $U'=(\tau', \tau'')$.

Consider the two connected open components of $\C\setminus R_{\tau'}\cup
\imp(\tau')\cup \imp(\tau'') \cup R_{\tau''}$; let $V$ be the component
containing rays of angles from $U'$. Then there are points of $J$ in $V$.
Indeed, otherwise $\imp(\tau')\cup \imp(\tau'')$ contains the impressions of
all angles from $U'$ which yields that a forward $\si$-image of $\tau'$ or
$\tau''$ will coincide with $J$ implying by Theorem~\ref{indec} that $J$ is
indecomposable, a contradiction with the standing assumption. Let us prove that
$\al$ is K-separate and its impression is a point. Indeed, by the previous
paragraph $V$ is positively distant from $B'$. Since $V$ is simply connected,
we can find two Jordan disks $W'\supset \overline{W''}\supset J\cap V$ which
are both positively distant from $B'$. Therefore all pull-backs of $W'$ and
$W''$ are univalent. By Ma\~n\'e \cite{mane93} this implies that the diameter
of the pull-backs of $W''$ converge to $0$ as the power of the map approaches
infinity. Observe that as $\si$-images of $\al$ approach $\ga$, the
corresponding $P$-images of $\imp(\al)$ get closer and closer to $\imp(\ga)$
(because of the upper semi-continuity of impressions) and thus we may assume
that infinitely many $P$-images of $\imp(\al)$ are contained in $W''$. Pulling
$W''$ back along the orbit of $\imp(\al)$ for longer and longer time we see
that the diameter of $\imp(\al)$ cannot be positive, and hence
$\imp(\al)=\{y\}$ is a point as claimed. Moreover, by Lemma~\ref{densenotcrit}
the angle $\al$ is K-separate. This completes the proof.

\end{proof}

We can now prove our main theorem.

\begin{theorem}\label{maina} For a uniCremer polynomial $P$ the following
facts are equivalent:

\begin{enumerate}

\item there is an impression not containing the Cremer point;

\item there is a degenerate impression;

\item the set $Y$ of all K-separate angles with degenerate
impressions contains all angles with dense orbits and a dense set of periodic
angles, and the Julia set $J$ is connected im kleinen at landing points of the
corresponding rays;

\item there is a point at which the Julia set is connected im
kleinen.

\end{enumerate}

\end{theorem}

\begin{proof} Let us prove that (1) implies (2). Indeed, suppose that
there is an angle not containing $p$ in its impression. Then by
Lemma~\ref{exishort} there exist angles with degenerate impressions.

We show that (2) implies (3). Indeed, let $\imp(\al)$ be a point. Then so
are the impressions of the angles $\al+1/d, \dots, \al+(d-1)/d$. At least one
of them is not $p$, so we may assume that $\imp(\al)\ne \{p\}$. Then by
Lemma~\ref{exishort} the set $E$ is dense in $\ucirc$. Let us now consider an
angle $\be$ whose $\si$-orbit is dense in $\ucirc$. By Lemma~\ref{transimp}
$\imp(\be)$ is a point and $\be$ is K-separate. By Lemma~\ref{condklei} $J$ is
connected im kleinen at the landing points of the rays with arguments either
from $E$, or from the set of angles with dense in $\ucirc$ orbits. This shows
that indeed (2) implies (3).

Clearly, (3) implies (4). It remains to show that (4) implies (1). Indeed, if
all impressions contain $p$ then by Lemma~\ref{redwarf} $J$ is nowhere
connected im kleinen, a contradiction. The proof is complete.
\end{proof}

\bibliographystyle{amsalpha}
\bibliography{/lex/references/refshort}
\providecommand{\bysame}{\leavevmode\hbox to3em{\hrulefill}\thinspace}
\providecommand{\MR}{\relax\ifhmode\unskip\space\fi MR }
\providecommand{\MRhref}[2]{%
  \href{http://www.ams.org/mathscinet-getitem?mr=#1}{#2}
} \providecommand{\href}[2]{#2}

\end{document}